\documentclass[11pt, reqno]{amsart}
\usepackage[utf8]{inputenc}
\numberwithin{equation}{section}
\usepackage{color}
\usepackage[shortlabels]{enumitem}
\usepackage[hidelinks]{hyperref}
\usepackage{cleveref}
\usepackage{todonotes}
\newcommand{\qtq}[1]{\quad\text{#1}\quad}
\usepackage{mathabx}
\usepackage{amsmath, amssymb, amsfonts, amsthm, mathtools}
\usepackage{dsfont}
\usepackage[scr=boondoxo]{mathalfa}

\theoremstyle{definition}
\newtheorem{definition}{Definition}

\theoremstyle{plain}
\newtheorem{theorem}{Theorem}
\newtheorem*{theorem*}{Theorem}
\newtheorem{lemma}{Lemma}

\theoremstyle{remark}
\newtheorem*{claim*}{Claim}

\newtheorem{remark}[]{Remark}
\newtheorem*{remark*}{Remark}

\newtheorem*{example*}{Example}

\newcommand{\eps}{\varepsilon}

\DeclareMathOperator{\R}{\mathbb{R}}

\DeclareMathOperator{\C}{\mathbb{C}}

\DeclareMathOperator{\bigO}{\mathcal{O}}
\DeclareMathOperator{\F}{\mathcal{F}}
\DeclareMathOperator{\K}{\mathcal{K}}
\DeclareMathOperator{\M}{\mathcal{M}}

\newcommand{\sW}{\mathscr{W}}

\DeclareMathOperator{\Sw}{\mathcal{S}}

\newcommand{\jbrak}[1]{\langle#1\rangle}

\renewcommand{\d}{\mathrm{d}}
\renewcommand{\:}{\colon}

\newcommand{\bbar}{\overline}

\renewcommand{\tilde}{\widetilde}
\renewcommand{\hat}{\widehat}

\renewcommand{\d}{\mathrm{d}}
\newcommand{\dxi}{\, \mathrm{d}\xi}
\newcommand{\dx}{\, \mathrm{d}x}
\newcommand{\dy}{\, \mathrm{d}y}
\newcommand{\ds}{\, \mathrm{d}s}

\newcommand{\dz}{\, \d z}
\newcommand{\dr}{\, \d r}

\newcommand{\cR}{\mathcal{R}}

\newcommand{\cD}{\mathcal{D}}

\newcommand{\defe}{\overset{\mathrm{def}}{=}}
\renewcommand{\rm}[1]{\mathrm{#1}}

\title[Modified Scattering for Nonlocal NLS]{Modified Scattering for Nonlocal Nonlinear Schr\"odinger Equations}
\author{Tim Van Hoose}

\begin{document}
    \begin{abstract}
        We prove a modified scattering and sharp $L^\infty$ decay result for both the Hartree and Schr\"odinger-Bopp-Podolsky equations in dimensions $2$ and $3$ using the testing by wavepackets approach due to Ifrim and Tataru \cite{ifrimGlobalBoundsCubic2014,ifrimTestingWavePackets2022a}. We show that modified scattering and sharp pointwise decay occur for these equations at a regularity much lower than previous results due to Hayashi-Naumkin and Kato-Pusateri \cite{hayashiAsymptoticsLargeTime1998a,katoNewProofLong2010a}, and as a corollary also show that the results on power-type scattering-critical NLS due to \cite{hayashiAsymptoticsLargeTime1998a} can be proven under minimal regularity assumptions. 
    \end{abstract}
\maketitle

\section{Introduction}
    We will consider the long-time behavior for suitably small solutions for two different nonlocal nonlinear Schr\"odinger equations: the Hartree equation, given by 
    \begin{equation}\label{E:VNLS}
        \begin{cases}
            i\partial_t u + \Delta u = (|\cdot|^{-1} \ast |u|^2)u \\
            u(0, x) = u_0(x)
        \end{cases}
        \quad \text{on} \R_t \times \R_x^d,
    \end{equation}
    and the Schr\"odinger-Bopp-Podolsky equation, given by 
    \begin{equation}\label{E:SBP}
        \begin{cases}
            i\partial_t u + \Delta u = (\K \ast |u|^2)u - |u|^{\frac{2}{d}}u \\
            u(0, x) = u_0(x) 
        \end{cases}\text{on} \R_t \times \R_x^d.
    \end{equation}
    In both cases, we assume that $d \in \{2, 3\}$. Here the kernel $\K$ is the Bopp-Podolsky potential, given by 
    \begin{equation*}
        \K(x) = \frac{1-e^{-|x|}}{|x|}.
    \end{equation*}
    Our primary goal is to prove a modified scattering result using the testing by wavepackets technique of Ifrim and Tataru \cite{ifrimGlobalBoundsCubic2014,ifrimTestingWavePackets2022a}. As a consequence, we obtain modified scattering well below the `classical' regularity of $H^{\frac{d}{2}+, \frac{d}{2}+}$; our result works for data in the class $H^{0, \frac{d}{2}+}$. Here, the space $H^{\gamma, \nu}(\R^d)$ is a weighted $L^2$-Sobolev space defined by the norm \eqref{E:wgtSobnorm}, and by $\frac{d}{2}+$ we mean $\frac{d}{2}+\eps$ for any $\eps > 0$.
    
    Before we state our main theorem, we fix once and for all the parameter 
    \begin{equation}
        \beta = \frac{d}{2}+\frac{1}{10}
    \end{equation}
    which governs the number of spatial weights which we use for the main result. With this notation in hand, we can now state our main theorem:
    \begin{theorem}\label{T:maintheorem}
            Let $d \in \{2, 3\}$, $0 < \eps \ll 1$, and $\|u_0\|_{H^{0, \beta}(\R^d)} =\eps$. Then there exists a unique global solution $u(t,x)$ to either \eqref{E:VNLS} or \eqref{E:SBP} belonging to $H^{0, \beta}$ in the sense that $e^{-it\Delta}u \in L_t^\infty H_x^{0, \beta}(\R \times \R^d)$. Furthermore the following additional properties hold for the solution $u$: 
            \begin{enumerate}
                \item Sharp $L^\infty$ decay and energy growth:
                    The solution $u$ satisfies the estimates 
                    \begin{equation} \label{E:sharpdecay}
                        \|u\|_{L^\infty} \lesssim \eps |t|^{-\frac{d}{2}}
                    \end{equation}
                    and
                    \begin{equation}\label{E:energygrowth}
                        \|e^{-it\Delta}u\|_{H^{0, \beta}} \lesssim 2\eps\jbrak{t}^{C\eps^{3-\frac{1}{d}}}.
                    \end{equation}
                \item Modified scattering: If $u$ is a solution to \eqref{E:VNLS}, then there exists a profile $\sW \in L^\infty(\R^d)$ so that (in the $L^\infty$ topology)
                \begin{equation}
                    u(t, x) = t^{-\frac{d}{2}}e^{i\frac{|x|^2}{4t}}\sW\left(\frac{x}{2t}\right)e^{-\frac{i}{2}\log(t)(|\cdot|^{-1} \ast |\sW(\frac{x}{2t})|^2)} + \mathcal{O}(t^{-\frac{d}{2}-\frac{1}{10d}}) \text{ as } t \to \infty.
                \end{equation}
                Further, if $u$ solves $\eqref{E:SBP}$, there exists a profile $\mathscr{Q} \in L^\infty(\R^d)$ so that in the $L^\infty$ topology 
                \begin{equation}
                    u(t,x) = t^{-\frac{d}{2}}e^{i\frac{|x|^2}{4t}}\mathscr{Q}\left(\frac{x}{2t}\right)e^{-\frac{i}{2}\log(t)\left(\K \ast |\mathscr{Q}(\frac{x}{2t})|^2) - |\mathscr{Q}(\frac{x}{2t})|^\frac{2}{d}\right)} + \mathcal{O}(t^{-\frac{d}{2}-\frac{1}{10d}}) \text{ as } t \to \infty.
                \end{equation}
            \end{enumerate}
    \end{theorem}
Our result substantially improves upon the modified scattering results of \cite{hayashiAsymptoticsLargeTime1998a,katoNewProofLong2010a} in the sense that we require only spatially localized data (as evidenced by the presence of spatial weights in the definition of the $H^{0, \nu}$-norm). Further, we are able to significantly sharpen the result of \cite{wadaAsymptoticExpansionSolution2001}, which also eliminates the requirement of spatial regularity, albeit at the cost of requiring additional spatial weight: their modified scattering result takes place in the regularity class $H^{0, (\frac{d}{2}+2)+}$, as compared to our result, which eliminates the extra $2$ spatial weights appearing above. Also, our techniques allow us to establish modified scattering in the aforementioned regularity class for the \emph{power-type} scattering-critical NLS, which was treated in \cite{hayashiAsymptoticsLargeTime1998a} and in the $1d$ cubic regime in \cite{katoNewProofLong2010a}. This effectively follows from the proof of modified scattering for \eqref{E:SBP} when one removes the nonlocal term. By simple modifications of the statements of the theorems, we are also able to treat all the way down to the endpoint $\frac{d}{2}$, which is sharp, owing to our need to use Sobolev embedding to turn our estimates in $H^{0, \beta}$ into pointwise estimates, \emph{viz} \eqref{E:GNarg} below; the main results of this paper are only stated in terms of the fixed parameter $\beta$ for clarity. 
    
The models we address both arise in electrodynamics: originally, the Hartree model was used to study many-electron atoms; the equation above is in some sense a limiting version of this model. On the other hand, the Bopp--Podolsky model above due (independently) to Bopp and Podolsky \cite{bopp_lineare_1940, podolsky_generalized_1942} was introduced to handle certain nonphysical divergences arising in classical electrodynamics; one can think of the potential appearing in \eqref{E:SBP} as a screened version of the Coulomb potential. For other works about the Bopp-Podolsky model, we direct the reader to the references \cite{gao_nonlinear_2023, zheng_existence_2022, davenia_nonlinear_2019}, among others. 

Our main result is a modified scattering result, showing that small initial data give rise to solutions with an explicit asymptotic expansion which is essentially scattering behavior (compare the expansion to e.g. the Fraunhofer formula for solutions to linear Schr\"odinger), but with the addition of a logarithmic phase correction, coming from the \textit{scattering criticality} of the equations. This criticality refers to the fact that in both \eqref{E:VNLS} and \eqref{E:SBP}, the nonlinear terms have growth rates which exactly match the decay expected by linear solutions to the Schr\"odinger equation; that is, for any Schwartz function $\phi$,
    \begin{align*}
        \| |\cdot|^{-1} \ast |e^{it\Delta}\phi|^2 \|_{L^\infty} &\lesssim_\phi |t|^{-1}\\
        \| \K \ast |e^{it\Delta}\phi|^2 \|_{L^\infty} &\lesssim_\phi |t|^{-1} \\
        \| |e^{it\Delta} \phi|^{\frac{2}{d}}\|_{L^\infty} &\lesssim_\phi |t|^{-1}.
    \end{align*}
In particular, this provides a (heuristic) obstruction to scattering; we would need these powers of $t$ to be integrable near infinity for our solution to scatter to a linear solution; however, as long as we account for an additional phase correction coming from the slow growth of the nonlinearity near infinity, we can still recover a similar asymptotic expansion. 
    
As previously indicated, our argument will use the method of testing by wavepackets, due to Ifrim and Tataru \cite{ifrimGlobalBoundsCubic2014,ifrimTestingWavePackets2022a}. Briefly, this method proceeds by testing the solution $u$ against a wavepacket $\Psi_v(t, x)$, which is effectively a Schwartz function which approximately solves the linear Schr\"odinger equation, up to a quantifiable error (see \Cref{L:WPsol}). After proving that the equation for $u$ is globally well-posed in $H^{0, \beta}$, we prove \Cref{L:gammabounds}, which explicitly shows the sense in which $\gamma$ approximates $u$. Finally, we derive an approximate ODE for $\gamma$; showing that the error in this ODE is integrable in time will in the end allow us to close a bootstrap argument, establishing both the energy growth and time decay claimed in the main theorem. 
    
After closing the bootstrap, we explicitly solve the ODE for $\gamma$ using an integrating factor, which after some manipulations and the bounds in \Cref{L:gammabounds} explicitly gives the asymptotic expansion claimed above. For more background on this method, or more context on methods for establishing modified scattering, we direct the reader to any of the references \cite{hayashiAsymptoticsLargeTime1998a,katoNewProofLong2010a,ifrimTestingWavePackets2022a,ifrimGlobalBoundsCubic2014,okamotoLongtimeBehaviorSolutions2017,murphyReviewModifiedScattering2021,wadaAsymptoticExpansionSolution2001}.

We also highlight another related work which uses the testing by wavepackets method, and may provide an alternate perspective to our work. In the work of \cite{cloos_long-time_2020}, the author considers long-time behavior for the Dirac-Maxwell equation, a similar nonlocal model to the ones considered here, albeit based on the Klein-Gordon equation rather than the Schr\"odinger equation.

    % Further, our result substantially improves upon the modified scattering results of \cite{hayashiAsymptoticsLargeTime1998a,katoNewProofLong2010a} in the sense that we require effectively only spatially localized data (as evidenced by the presence of spatial weights in the definition of the $H^{0, \nu}$-norm). Further, we are able to substantially sharpen the result of \cite{wadaAsymptoticExpansionSolution2001}, which also eliminates the requirement of spatial regularity, albeit at the cost of requiring additional spatial weight: their modified scattering result takes place in the regularity class $H^{0, (\frac{d}{2}+2)+}$, as compared to our result, which eliminates the extra $2$ spatial weights appearing above. Also, our techniques allow us to establish modified scattering in the aforementioned regularity class for the \emph{power-type} scattering-critical NLS, which was treated in \cite{hayashiAsymptoticsLargeTime1998a} and in the $1d$ cubic regime in \cite{katoNewProofLong2010a}. This effectively follows from the proof of modified scattering for \eqref{E:SBP} when one removes the nonlocal term. By simple modifications of the statements of the theorems, we are also able to treat all the way down to the endpoint $\frac{d}{2}$, which is sharp, owing to our need to use Sobolev embedding to turn our estimates in $H^{0, \beta}$ into pointwise estimates, \emph{viz} \eqref{E:GNarg} below; the main results of this paper are only stated in terms of the fixed parameter $\beta$ for clarity. 

We organize the remainder of the paper as follows: in \Cref{S:notation}, we compile some notation, lemmas and definitions for quantities that will be used in the remainder of the paper. In \Cref{S:GWP}, we will prove the global well-posedness result we claimed earlier and the bounds along rays used to prove the first part of \Cref{T:maintheorem}. Finally, in \Cref{S:asymptotics}, we will use the results of \Cref{S:GWP} to write down the asymptotic expansions appearing as the second part of \Cref{T:maintheorem}.
    \subsection*{Acknowledgements} We are thankful to Jeremy Marzuola for his guidance and support, for many helpful conversations about the problem, and his careful reading of the paper. T.V.H. was supported by the Simons Dissertation Fellowship and by NSF Grant DMS-2307384.

\section{Notation}\label{S:notation}
    In this section, we introduce some notation that will be used throughout the remainder of the paper.\par
    We write $A \lesssim B$ or $B \gtrsim A$ to denote the inequality $A \leq CB$ for some constant $C > 0$, where $C$ may depend on parameters like the dimension or the indices of function spaces. If $A \lesssim B $ and $B \lesssim A$ both hold, then we write $A \sim B$. We will also make use of the standard Landau symbol $\bigO$, as well as the Japanese bracket notation $\jbrak{\cdot} := (1+|\cdot|^2)^{\frac{1}{2}}$. \par
    We use the standard Lorentz spaces $L^{p,q}(\R^d)$, specifically the fact that $|\cdot|^{-1} \in L^{d, \infty}(\R^d)$. These spaces are defined by the (quasi)norm 
    \begin{equation*}
        \|f\|_{L^{p,q}(\R^d)} :=
        \begin{cases}
                   p^{\frac{1}{q}} \left(\displaystyle\int_0^\infty t^q \ m\{x \in \R^d \: |f(x)| \geq \lambda\}^{\frac{q}{p}}\frac{\d \lambda} {\lambda}\right)^{\frac{1}{q}}, q< \infty \\
                   \displaystyle\sup_{\lambda >0} \lambda^p \ m\{x \in \R^d \: |f(x)| \geq \lambda\}, q = \infty
        \end{cases}     
    \end{equation*}
    where $m$ is $d$-dimensional Lebesgue measure.

    % For our local well-posedness result (see \Cref{SS:LocalTheory}), we need the standard (endpoint) Strichartz estimates for the Schr\"odinger equation proven in \cite{taoNonlinearDispersiveEquations2006c} and \cite{keelEndpointStrichartzEstimates1998}.
    
    We will denote the Fourier transform by $\F[f](\xi) = \widehat{f}(\xi)$ with the normalization 
    \begin{equation*}
        \F[f](\xi) := (2\pi)^{-\frac{d}{2}} \int_{\R^d} e^{-ix \cdot\xi} f(x) \dx 
    \end{equation*}
    and inverse 
    \begin{equation}
        \F^{-1}[g](x) = \check{g}(x) := (2\pi)^{-\frac{d}{2}} \int_{\R^d} e^{ix\cdot\xi} g(\xi)\d \xi.
    \end{equation}
    We will define the weighted Sobolev spaces $H^{\gamma, \nu}(\R^d)$ by the norm
    \begin{equation}\label{E:wgtSobnorm}
        \|u\|_{H^{\gamma, \nu}} := \|\jbrak{\nabla}^\gamma u\|_{L^2} + \| |x|^\nu u\|_{L^2}
    \end{equation}
    where as usual $\jbrak{\nabla}^\gamma := \F^{-1} \jbrak{\xi}^\gamma \F$, and we define the standard Sobolev spaces $H^s(\R^d):= H^{s, 0}(\R^d)$ in the notation above. \par
    In the usual way, we will denote the free Schr\"odinger propagator by $e^{it\Delta}$. Direct computation shows that we can decompose this operator as 
    \begin{equation}\label{E:MDFM}
        e^{it\Delta} = \M(t) \cD(t) \F \M(t)
    \end{equation}
    where 
    \begin{equation}
        \M(t)f(x) = e^{i\frac{|x|^2}{4t}} f(x) \qquad \text{and} \qquad \cD(t) = (2it)^{-\frac{d}{2}} f\left(\frac{x}{2t}\right).
    \end{equation}
    By direct computation, we see that 
    \begin{equation}
        \cD(t)^{-1} = (2i)^{d} \cD\left(\frac{1}{t}\right) .
    \end{equation}
    We will also make use of the Galilean operator $J(t):= x + 2it\nabla$. Direct computation shows that
    \begin{equation}\label{E:Jtdefn1}
        J(t) = \M(t) (2it\nabla) \M(-t).
    \end{equation}
    Indeed, we have 
    \begin{align*}
        \M(t)(2it\nabla)\M(-t) f &= e^{i\frac{|x|^2}{4t}}(2it\nabla)e^{-i\frac{|x|^2}{4t}}f \\
        &= e^{i\frac{|x|^2}{4t}}e^{-i\frac{|x|^2}{4t}} (x+2it\nabla)f \\
        &= J(t)f.
    \end{align*}
    An ODE argument furnishes the identity 
    \begin{equation}\label{E:Jtdefn2}
        J(t) = e^{it\Delta}xe^{-it\Delta}.
    \end{equation}
    Indeed, both sides of the equation match at $t=0$. If we take a time derivative on both sides, we find that 
    \begin{align*}
        J'(t) &= 2i\nabla \\
        \frac{d}{dt} [e^{it\Delta} x e^{-it\Delta}] &= i e^{it\Delta} [\Delta, x] e^{-it\Delta}
    \end{align*}
    where $[\cdot, \cdot]$ is the usual operator commutator. One can check directly that $[\Delta, x] = 2\nabla$. Since Fourier multipliers commute and $(e^{it\Delta})_{t \in \R}$ is a semigroup, we see 
    \begin{equation*}
        i e^{it\Delta} [\Delta, x] e^{-it\Delta} = 2i\nabla.
    \end{equation*} 
    Since the time derivatives match for all $t$ and the two expressions have matching values at $t=0$, ODE uniqueness implies that $J(t) = e^{it\Delta}x e^{-it\Delta}$ for all $t$, which was what we claimed. 
    
    \begin{remark*}
        Since $e^{it\Delta}: L^2 \to L^2$ is unitary, we see that $ \| J(t)u\|_{L^2} = \|x e^{-it\Delta}u\|_{L^2}.$
    \end{remark*}
    We also define powers of $J(t)$ in the following manner:
    \begin{align}
        |J|^{\gamma}(t) &:= \M(t)(-4t^2 \Delta)^{\frac{\gamma}{2}}\M(-t) \quad \text{for }\gamma\in[0, \infty) \label{E:Jtpower}\\
        &= e^{it\Delta}|x|^\gamma e^{-it\Delta}. \label{E:Jtpower2}
    \end{align}
    Finally, we will need an abstract interpolation result from \cite{berghInterpolationSpacesIntroduction1976}, which we have specialized to our particular case. This will allow us to directly estimate convolutions against $|x|^{-1}$ in $L^\infty$, which is an inadmissible endpoint for Hardy-Littlewood-Sobolev. 
    
    \begin{lemma}[Real Interpolation of $L^p$ spaces]\label{L:interplemma}
        If $1 \leq p_0 \neq p_1 \leq \infty$, then we have the equality of spaces (with equivalent norms) 
        \begin{equation}
            [L^{p_0}, L^{p_1}]_{\theta, q} = L^{p, q},
        \end{equation}
        where $\frac{1}{p} = \frac{1-\theta}{p_0} + \frac{\theta}{p_1}$ and $0 < \theta < 1$. Further, it holds that 
        \begin{equation}
            \|u\|_{L^{p, q}} \lesssim \|u\|_{L^{p_0}}^{1-\theta} \|u\|_{L^{p_1}}^\theta.
        \end{equation}
    \end{lemma}
    \begin{remark}
        The first part of this lemma is directly from \cite[Theorem 5.3.1]{berghInterpolationSpacesIntroduction1976}. The second half follows from the discussion in \cite[Section 3.5.1f]{berghInterpolationSpacesIntroduction1976}. In our particular instance, this lemma will be applied to the spaces $L^2$ and $L^\infty$. To be completely explicit, we will use that 
        \begin{equation*}
            \|u\|_{L^{\frac{2d}{d-1}, 2}} \lesssim \|u\|_{L^2}^{\frac{d-1}{d}} \|u\|_{L^\infty}^{\frac{1}{d}},
        \end{equation*}
        since $\frac{2d}{d-1} > 2$, and the exponents come directly from \Cref{L:interplemma}. For more explicit details, see \Cref{AS:InterpAppendix}.
    \end{remark}
    Next, we record some of the regularity properties of $\K(x)$ that will be of use to us later.
        \begin{lemma}\label{L:Kregularity}
            Let $\K(x) : \R^d \to \R$ be defined as above. Then 
            \begin{enumerate}
                \item $\K \in L^{d, \infty}(\R^d)$ and $\K \in L^p(\R^d)$ for all $d < p \leq \infty$. 
                \item The following convolution estimate for $\K$ holds:
                \begin{equation*}
                    \|\K \ast f\|_{L^r} \lesssim \|f\|_{L^q}
                \end{equation*}
                where $q < r$ and $1-\tfrac{1}{d} \leq \tfrac{1}{q}-\tfrac{1}{r} \leq 1$.
            \end{enumerate}
        \end{lemma}
        \begin{proof}
            To see the first part of (1), note that we have $\K(x) \leq |x|^{-1}$ pointwise in $x$. This readily implies that $\K(x)$ is in $L^{d, \infty}(\R^d)$ because $x \mapsto |x|^{-1} \in L^{d, \infty}(\R^d)$. For the second part of the (1), we directly integrate:
            \begin{align*}
                \|\K(x)\|_{L^p}^p &= \int_{\R^d} \left(\frac{1-e^{-|x|}}{|x|}\right)^p \dx \\
                &\leq \int_{B(0, 1)} \left(\frac{1-e^{-|x|}}{|x|}\right)^p \dx + \int_{\R^d \setminus B(0, 1)} \left(\frac{1-e^{-|x|}}{|x|}\right)^p \dx \\
                &\lesssim 1 + \int_{\R^d \setminus B(0, 1)} \left(\frac{1-e^{-|x|}}{|x|}\right)^p \dx\\
                &\lesssim 1+ \int_{\R^d \setminus B(0, 1)}  |x|^{-p} \dx \\
                &\lesssim 1+ \int_{S^{d-1}}\int_1^\infty r^{d-1-p} \dr \d \Omega,
            \end{align*}    
            where from the second to the third line we used that $\K \in L^\infty$, and from the third to the fourth line we used that $\K(x) \leq |x|^{-1}$. In the final line we switched to spherical coordinates. In particular, the last integral in spherical coordinates converges if and only if $d - p < 0$, i.e. if $p > d$. \par
            To prove (2), note that by Young's convolution inequality 
            \begin{equation}\label{E:youngs}
                \|\K \ast f\|_{L^r} \lesssim \|\K\|_{L^p}\|f\|_{L^q}
            \end{equation}
            where $1+\tfrac{1}{r} = \tfrac{1}{p}+ \tfrac{1}{q}$ for any $d < p \leq \infty$. Rearranging the former equality gives $\tfrac{1}{q}-\tfrac{1}{r} = 1-\tfrac{1}{p}$. The conditions on $p$ then guarantee that $\tfrac{1}{q}-\tfrac{1}{r} \in (1-\tfrac{1}{d}, 1]$. To conclude, we note that a stronger version of Young's inequality holds, namely that we can replace the $L^p$-norm in \eqref{E:youngs} by $L^{p, \infty}$. This allows us to include $d$ in the range of $p$ for which the inequality holds, and thus include the endpoint $1-\tfrac{1}{d}$. 
        \end{proof}
    \begin{remark}[Convention for convolutions]\label{R:convolution}
        We will frequently work with the operation of convolution with the kernel $|\cdot|^{-1}$. For notational ease, we will write $|\cdot|^{-1} \ast f(t, z)$ to mean the convolution with output variable $z$: 
    \begin{equation*}
        |\cdot|^{-1} \ast f(t,z) := \int_{\R^d} \frac{f(t,y)}{|z-y|} \dy.
    \end{equation*}
    For example, we will often see convolution with functions of the variable $2tv$. We interpret this in light of the definition above, with the replacement $z \mapsto 2tv$. If the argument is suppressed, we take that to mean that the output is the variable $x$. 
    \end{remark}
    \section{Global Existence and Wavepacket Estimates}\label{S:GWP}
    In this section, we will prove the claimed global well-posedness result, along with the estimates along rays that will later prove the first part of \Cref{T:maintheorem}.
        To begin, we prove that \eqref{E:VNLS} and \eqref{E:SBP} are globally well-posed in $H^{0, \beta}$; we will do this by a standard $L^2$ local well-posedness result combined with a persistence of regularity argument.
        \begin{theorem}\label{T:GWP}
            Let $d \in \{2, 3\}$ and set $\beta = \frac{d}{2}+\frac{1}{10}$. Then the equations \eqref{E:VNLS} and \eqref{E:SBP} are globally well-posed in $H^{0, \beta}$, in the sense that they have a unique solution $u \in C_t L_x^2(\R \times \R^d)$ with $|J|^\beta u \in C_tL_x^2(\R \times \R^d)$. Any such solution $u(t, x)$ is in $C_t L_x^\infty$, and near $t = 0$ we have 
            \begin{equation}
                \|u(t, x)\|_{L_x^\infty} \lesssim t^{-\frac{d}{2}} \|u_0\|_{H^{0, \beta}}.
            \end{equation}
        \end{theorem}
        \begin{proof}
            To prove the $L^2$ well-posedness, we rephrase the problem using the Duhamel formula, leading to the equation 
            \begin{equation}
                \Phi[u] := u(t,x) = e^{it\Delta}u_0 -i \int_0^t e^{i(t-s)\Delta}\{(|x|^{-1} \ast |u|^2)u(s)\} \ds.
            \end{equation}
            We have to treat the cases $d = \{2, 3\}$ with different Strichartz norms, but the arguments are entirely identical. To this end, define 
            \begin{equation*}
                A_2 = L_{T,x}^4(\R^2) \qtq{and} A_3 = L_T^5 L_x^{\frac{30}{11}}(\R^3).
            \end{equation*}
            We will run a contraction mapping argument in the space 
            \begin{equation}
                X_{T, d} = \left\{u\: [0, T] \times \R^d \to \C \mid \|u\|_{L_T^\infty L_x^2} \lesssim 2\|u_0\|_{L^2} \text{ and } \|u\|_{A_d}\lesssim 2C\|u_0\|_{L^2}\right\}
            \end{equation}
            endowed with the metric $d(u, v) = \|u-v\|_{L_T^\infty L_x^2}$, and where $C$ encodes all the arbitrary constants that come from Strichartz estimates and Hardy-Littlewood-Sobolev. For the sake of brevity, we present only the key nonlinear estimates here:
            \begin{align*}
                \left\| \int_0^t e^{i(t-s)\Delta}(|x|^{-1} \ast |u|^2)u(s)\ds\right\|_{L_T^\infty L_x^2(\R^2)} &\lesssim \| (|\cdot|^{-1} \ast |u|^2)u(s)\|_{L_{T, x}^{\frac{4}{3}}} \\
                &\lesssim \|u\|_{L_T^\infty L_x^2} \| |\cdot|^{-1} \ast |u|^2\|_{L_T^{\frac{4}{3}}L_x^4} \\
                &\lesssim \|u\|_{L_T^\infty L_x^2} \||u|^2\|_{L_{T, x}^{\frac{4}{3}}} \\
                &\lesssim T^{\frac{1}{2}}\|u\|_{L_T^\infty L_x^2}^2 \|u\|_{A_2},
            \end{align*}
            which is acceptable. The estimate of the $A_2$ component of the norm is entirely identical, by Strichartz. In three spatial dimensions, we have the nonlinear estimate
            \begin{align*}
                \left\| \int_0^t e^{i(t-s)\Delta} (|x|^{-1} \ast |u|^2)u(s)\ds\right\|_{L_t^\infty L_x^2(\R^3)} &\lesssim \| (|x|^{-1} \ast |u|^2)u\|_{L_{T, x}^{\frac{10}{7}}}\\
                &\lesssim \|u\|_{L_T^\infty L_x^2}\| |x|^{-1} \ast |u|^2\|_{L_T^{\frac{10}{7}}L_x^5} \\
                &\lesssim \|u\|_{L_T^\infty L_x^2}\||u|^2\|_{L_T^{\frac{10}{7} }L_x^{\frac{15}{13}}}\\
                &\lesssim T^{\frac{1}{2}}\|u\|_{L_T^\infty L_x^2}^2 \|u\|_{A_3},
            \end{align*}
            where to go from the second to the third line we applied the Hardy-Littlewood-Sobolev inequality. Using the definition of the space $X_{T, d}$, we see that for $T$ small enough depending only on $\|u_0\|_{L^2}$, we have a map from $X_{T, d} \to X_{T, d}$ defined by the right-hand side of the Duhamel formula above. To prove that $\Phi$ is a contraction map, we consider the difference $\|\Phi[u]-\Phi[v]\|_{L_T^\infty L_x^2}$. To estimate this, we see that once again the crucial term is the integral term of the Duhamel formula. To estimate it, we rewrite the difference of convolutions as 
            \begin{multline}
                (|\cdot|^{-1} \ast |u|^2)u - (|\cdot|^{-1} \ast |v|^2)v = (|\cdot|^{-1} \ast |u|^2)(u-v) \\+ (|\cdot|^{-1} \ast (u-v)\bbar{u})v+ (|\cdot|^{-1} \ast (\bbar{u}- \bbar{v})v)v.
            \end{multline}
            Checking that $\Phi$ is a contraction mapping for $T$ sufficiently small proceeds in an entirely analogous way to the estimate above. Hence there is a unique $L^2$ solution to \eqref{E:VNLS}; by conservation of mass, it is global. \par
            We would now like to upgrade to $H^{0, \beta}$ initial data; this will follow from a standard persistence of regularity argument. To do this, we note that the operator $|J|^\beta$ defined above commutes nicely with the linear part of the equation. By Strichartz, the linear term is easily controlled in terms of $\|u_0\|_{H^{0, \beta}}$; it remains to handle the nonlinear term, which takes the form 
            \begin{equation}\label{E:persistenceintegral}
                \M(t) (2it)^\beta |\nabla|^\beta \M(-t) \int_0^t e^{i(t-s)\Delta} (|\cdot|^{-1} \ast |u|^2)u \ds.
            \end{equation}
            Estimating this in $L_T^\infty L_x^2$ and applying Strichartz, we see that it suffices to provide good estimates for 
            \begin{equation}
               \| |\nabla|^\beta [(|\cdot|^{-1} \ast |w|^2)w]\|_{L_{T, x}^{\frac{4}{3}}} \qtq{and} \| |\nabla|^{\beta} [(|\cdot|^{-1} \ast |w|^2)w]\|_{L_{T, x}^{\frac{10}{7}}},
            \end{equation}
            where we write $w = \M(-t)u$. By the fractional chain rule, we can control the terms above by
            \begin{align*}
                \| |\nabla|^\beta [(|\cdot|^{-1} \ast |w|^2)w]\|_{L_{T, x}^{\frac{4}{3}}} &\lesssim
                \begin{multlined}[t]
                    \|w\|_{L_T^\infty L_x^2} \| |\cdot|^{-1} \ast |\nabla|^\beta |w|^2\|_{L_T^\frac{4}{3}L_x^4} \\
                    + \||\nabla|^\beta w \|_{L_T^\infty L_x^2} \| |\cdot|^{-1} \ast |w|^2\|_{L_T^{\frac{4}{3}}L_x^4}
                \end{multlined}\\
                &\lesssim T^{\frac{1}{2}} \|u\|_{L_T^\infty L_x^2} \|u\|_{A_2} \| |\nabla|^\beta w\|_{L_T^\infty L_x^2}
            \end{align*}
            in two spatial dimensions, and in three spatial dimensions, we have 
            \begin{align*}
                \| |\nabla|^{\beta} [(|\cdot|^{-1} \ast |w|^2)w]\|_{L_{T, x}^{\frac{10}{7}}} &\lesssim \begin{multlined}[t]
                    \|u\|_{L_T^\infty L_x^2} \| |\cdot|^{-1} \ast |\nabla|^\beta |u|^2\|_{L_T^{\frac{10}{7}} L_x^5} \\
                    +\| |\cdot|^{-1} \ast |u|^2 \|_{L_T^{\frac{10}{7}}L_x^5}\| |\nabla|^{\beta}u\|_{L_T^\infty L_x^2}.
                \end{multlined}\\
                &\lesssim T^{\frac{1}{2}}\|u\|_{L_T^\infty L_x^2} \|u\|_{A_3}\||\nabla|^{\beta}w\|_{L_T^\infty L_x^2}.
            \end{align*}
            Reinserting this into \eqref{E:persistenceintegral} and recalling the definition of $|J|^\beta$, we see that 
            \begin{equation}
                \| |J|^{\beta}u\|_{L_T^\infty L_x^2} \lesssim \|u_0\|_{H^{0, \beta}} + \frac{1}{2}\||J|^\beta u\|_{L_T^\infty L_x^2},
            \end{equation}
            where by the contraction mapping argument from earlier, we can choose $T$ depending only on $\|u_0\|_{L^2}$ to make the constant in front of $|J|^\beta u$ on the right-hand side equal to $\frac{1}{2}$. This implies that on the interval $[0, T]$ the norm of $|J|^\beta u$ grows by no more than a factor of $2$ in terms of the initial data. Again invoking mass conservation, we see that $\||J|^\beta u\|_{L^2}$ is finite along the global flow. \par
            Finally, we would like to obtain $L^\infty$ bounds for the solution. This follows from a virtually identical argument to \cite{ifrimGlobalBoundsCubic2014}; we reproduce it in our case for posterity. First, note (with $w$ defined the same way as earlier) that we have the identity
            \begin{equation}\label{E:jbetaidentity}
                \M(-t) |J|^\beta u = (2it)^\beta |\nabla|^\beta w.
            \end{equation}
            In particular, we see immediately that $w \in H^\beta$, and thus in $L^\infty$; to conclude, we use Gagliardo-Nirenberg:
            \begin{equation}\label{E:GNarg}
                \|u(t)\|_{L^\infty} = \|w(t)\|_{L^\infty} \lesssim \|w(t)\|_{L^2}^{1-\frac{d}{2\beta}}\||\nabla|^\beta w\|_{L^2}^{\frac{d}{2\beta}}. 
            \end{equation}
            The right-hand side is easily seen to be bounded by $t^{-\frac{d}{2}}\|u_0\|_{H^{0, \beta}}$ by rearranging \eqref{E:jbetaidentity} and using the properties of the solution $u$.
            
            To see that the same global existence argument holds for \eqref{E:SBP}, we note that all the same bounds hold for the convolution nonlinearity $(\K \ast |u|^2)u$, so that for the local existence argument we only need to handle the term $\| |u|^{\frac{2}{d}}u\|$ in the relevant spacetime norms. To this end, note that 
            \begin{equation*}
                \left \| \int_0^t e^{i(t-s)\Delta} |u|u \ds \right \|_{L_T^\infty L_x^2(\R^2)} \lesssim \| |u| u\|_{L_{T,x}^{\frac{4}{3}}} \lesssim T^{\frac{1}{2}}\|u\|_{L_T^\infty L_x^2} \|u\|_{A_2},
            \end{equation*}
            and 
            \begin{align*}
                \left \| \int_0^t e^{i(t-s)\Delta} |u|^{\frac{2}{3}}u \ds \right \|_{L_T^\infty L_x^2(\R^3)} &\lesssim \| |u|^{\frac{2}{3}} u\|_{L_t^{10/7}L_x^{10/7}} \\ 
                &\lesssim \| |u|^\frac{2}{3}\|_{L_t^\infty L_x^3} \| u \|_{L_t^{10/7} L_x^{30/11}} \\
                &\lesssim T^{\frac{1}{2}}\| u\|_{L_t^\infty L_x^{2}}^{\frac{2}{3}} \|u\|_{L_t^{5}L_x^{30/11}},
            \end{align*}
            which is acceptable, in view of the definitions above. Rerunning the same persistence of regularity argument using the definition of $|J|^\beta$ from above establishes the desired result.
        \end{proof}
        \begin{definition}
            Let $\vartheta \in \Sw(\R^d)$ have total integral $1$. Then a wave packet with velocity $v$ is defined to be 
            \begin{equation*}
                \Psi_{v}(t, x) := \vartheta\left(\frac{x -2tv}{t^{\frac{1}{2}}}\right)e^{i\frac{|x|^2}{4t}}.
            \end{equation*}
        \end{definition}
        The reason for the rescaling $t^{-\frac{1}{2}}$ of the argument of $\vartheta$ is so that the wavepacket lives at the scale dictated by the Heisenberg uncertainty principle. \par
        In fact, for any $v \in \R^d$, $\Psi_v(t,x)$ is an approximate solution to the linear Schr\"odinger equation in the following sense:
        \begin{lemma}\label{L:WPsol}
            In $d$ dimensions, we have for any $v \in \R^d$
            \begin{equation}\label{E:approxsoln}
                (i\partial_t + \Delta)\Psi_v(t, x) = \frac{1}{2t}e^{i\frac{|x|^2}{4t}} \nabla \cdot \bigg\{
                \begin{multlined}[t]
                    i(x-2vt)\vartheta\left(\frac{x-2vt}{t^\frac{1}{2}}\right) \\+ 2t^{\frac{1}{2}}\nabla\vartheta\left(\frac{x-2vt}{t^\frac{1}{2}}\right)\bigg\}
                \end{multlined}
            \end{equation}
        \end{lemma}
        \begin{proof}
            This follows by a direct (albeit involved) computation. To wit, we have 
            \begin{equation*}
                i\partial_t \Psi_v(t, x) = \frac{|x|^2}{4t^2}e^{i\frac{|x|^2}{4t}} \vartheta\left(\frac{x -2tv}{t^{\frac{1}{2}}}\right) + i\left(-\frac{1}{2}xt^{-\frac{3}{2}} - \frac{2vt^{-\frac{1}{2}}}{2}\right)\cdot (\nabla\vartheta)\left(\frac{x -2tv}{t^{\frac{1}{2}}}\right)e^{i\frac{|x|^2}{4t}}.
            \end{equation*}
            Next, we need to handle the spatial derivatives using 
            \[
                \Delta(fg) = f \Delta g + 2(\nabla f) \cdot (\nabla g) + g \Delta f.    
            \]
            Direct computation yields 
            \begin{equation*}
                \nabla e^{i\frac{|x|^2}{4t}} = \frac{ix}{2t}e^{i\frac{|x|^2}{4t}}
            \end{equation*}
            and 
            \begin{equation*}
                \nabla \vartheta\left(\frac{x -2tv}{t^{\frac{1}{2}}}\right) = \frac{1}{t^\frac{1}{2}}(\nabla\vartheta)\left(\frac{x -2tv}{t^{\frac{1}{2}}}\right). 
            \end{equation*}
            To compute the second derivatives, we use the product rule for gradients and divergences. This yields
            \begin{equation}
                \Delta e^{i\frac{|x|^2}{4t}} = \frac{id}{2t}e^{i\frac{|x|^2}{4t}} - \frac{|x|^2}{4t^2} e^{i \frac{|x|^2}{4t}}
            \end{equation}
            and 
            \begin{equation}
                \Delta \vartheta\left(\frac{x-2tv}{t^\frac{1}{2}}\right) = \frac{1}{t} (\Delta \vartheta)\left(\frac{x-2tv}{t^\frac{1}{2}}\right).
            \end{equation}
            Putting everything together, we have
            \begin{equation}\label{E:wavepacketpde1}
                \begin{aligned}
                    (i\partial_t + \Delta) \Psi_v (t, x) &= 
                    \frac{|x|^2}{4t^2}e^{i\frac{|x|^2}{4t}} \vartheta\left(\frac{x -2tv}{t^{\frac{1}{2}}}\right) \\
                    &+ i\left(-\frac{1}{2}xt^{-\frac{3}{2}} - \frac{2vt^{-\frac{1}{2}}}{2}\right)\cdot (\nabla\vartheta)\left(\frac{x -2tv}{t^{\frac{1}{2}}}\right)e^{i\frac{|x|^2}{4t}} \\
                    &+ \left(\frac{id}{2t}e^{i\frac{|x|^2}{4t}} - \frac{|x|^2}{4t^2} e^{i \frac{|x|^2}{4t}}\right)\vartheta\left(\frac{x -2tv}{t^{\frac{1}{2}}}\right) \\ 
                    &+ \frac{ix}{t^\frac{3}{2}} e^{i\frac{|x|^2}{4t}} \cdot (\nabla \vartheta)\left(\frac{x -2tv}{t^{\frac{1}{2}}}\right) \\
                    &+ \frac{1}{t} (\Delta\vartheta)\left(\frac{x -2tv}{t^{\frac{1}{2}}}\right) e^{i \frac{|x|^2}{4t}},
                \end{aligned}
            \end{equation}
            where the first line contains the time derivative and the second contains the space derivatives. One may simplify this expression somewhat with some simple algebra: 
            \begin{equation}
                \mathrm{RHS}\eqref{E:wavepacketpde1} =
                \begin{multlined}[t]
                    i\left(\frac{x-2vt}{2t^{\frac{3}{2}}}\right)\cdot (\nabla\vartheta)\left(\frac{x -2tv}{t^{\frac{1}{2}}}\right)e^{i\frac{|x|^2}{4t}} \\+\frac{id}{2t}e^{i\frac{|x|^2}{4t}}\vartheta\left(\frac{x -2tv}{t^{\frac{1}{2}}}\right) 
                    + \frac{1}{t}(\Delta \vartheta) \left(\frac{x-2vt}{t^\frac{1}{2}}\right)e^{i\frac{|x|^2}{4t}}.
                \end{multlined} 
            \end{equation}
        \end{proof}
        The point of this lemma is that the piece of \eqref{E:approxsoln} in brackets has exactly the same localization properties as $\vartheta$, but decays a factor of $\tfrac{1}{t}$ better than the original wavepacket. \par
        Next, define 
        \begin{equation*}
            \gamma(t, v) = \int_{\R^d} u(t, x) \bbar{\Psi_v}(t, x) \dx.
        \end{equation*}
        This quantity measures the decay of the solution $u$ along the ray $\Gamma_v$. Using Plancherel's theorem we can rewrite this expression as 
        \begin{equation*}
            \gamma(t, v) = \int_{\R^d} \hat{u}(t, \xi) \bbar{\hat{\Psi_v}}(t, \xi) \dxi.
        \end{equation*}
        One can compute the Fourier transform of $\Psi_v$ explicitly. Indeed, we have
        \begin{align*}
            \hat{\Psi_v}(t, \xi) &= (2\pi)^{-\frac{d}{2}}\int_{\R^d} e^{-ix \cdot \xi} e^{i\frac{|x|^2}{4t}} \vartheta\left(t^{-\frac{1}{2}}(x-2vt)\right)\dx \\
            &= (2\pi)^{-\frac{d}{2}}\int_{\R^d} e^{-i(x-2vt+2vt)\cdot\xi} e^{i \frac{|x-2vt+2vt|^2}{4t}}\vartheta(t^{-\frac{1}{2}}(x-2vt))\dx\\
            &=
                (2\pi)^{-\frac{d}{2}}e^{-2ivt\cdot \xi}\int_{\R^d}e^{-i(x-2vt)\cdot\xi} e^{i\frac{|x-2vt|^2}{4t}+i(x-2vt)\cdot v+ it|v|^2}\vartheta(t^{-\frac{1}{2}}(x-2vt))\dx
            \\
            &= 
                (2\pi)^{-\frac{d}{2}}e^{-it|\xi|^2}e^{it|\xi - v|^2} \int_{\R^d}e^{-i(x-2vt)\cdot(\xi - v)}e^{i\frac{|x-2vt|^2}{4t}}\vartheta(t^{-\frac{1}{2}}(x-2vt))\dx
            \\
            &= t^{\frac{d}{2}}e^{-it|\xi|^2} \tilde{\vartheta}(t^{\frac{1}{2}}(\xi - v)),
        \end{align*}
        where the function $\tilde{\vartheta}(\xi)$ is given by 
        \begin{equation*}
            \tilde{\vartheta}(\xi) = e^{i|\xi|^2}\F\left[e^{i\frac{|x|^2}{4}}\vartheta(x)\right](\xi).
        \end{equation*}
        Direct computation shows that in fact 
        \begin{equation*}
            \int_{\R^d} \tilde{\vartheta}(\xi) \dxi = \int_{\R^d} \vartheta(x) \dx =1. 
        \end{equation*}
        If we remember where we started, we had 
        \begin{align*}
            \gamma(t, v) &= \int_{\R^d} u(t, x) \bbar{\Psi_v}(t,x) \dx \\
            &= \int_{\R^d}t^{\frac{d}{2}}\hat{u}(t, \xi) \bbar{e^{-it|\xi|^2} \tilde{\vartheta}(t^{\frac{1}{2}}(\xi - v))}\dxi.
        \end{align*}
        Using the fact that the complex conjugate in the definition of $\gamma$ introduces a minus sign in the argument of $\hat{\Psi_v}$, we see  
        \begin{equation}\label{E:gammafreqrepn}
            \gamma(t, \xi) = e^{it|\xi|^2} \hat{u}(t, \xi) \ast_{\xi} t^{\frac{d}{2}}\tilde{\vartheta}(t^{\frac{1}{2}}\xi).
        \end{equation}
        We'd now like to compare $\gamma(t, v)$ to a solution $u(t, x)$ of \eqref{E:VNLS} along a ray $\Gamma_v$. Since we'll need it, note that by direct computation, we have the following equality in the sense of Fourier multipliers: 
        \begin{equation}
            |\nabla_v|^s = (2t)^s |\nabla_x|^s \qtq{for all} s \in \R.
        \end{equation}
        where we recall that $v = \tfrac{x}{2t}$ (in particular, the units are consistent across the equals sign). With these definitions in hand, we can state the lemma.
        \begin{lemma}\label{L:gammabounds}
            Set $\beta = \frac{d}{2}+\frac{1}{10}$. Then the function $\gamma(t, v)$ satisfies the bounds
            \begin{equation}
                \|\gamma\|_{L^\infty} \lesssim t^{\frac{d}{2}}\|u\|_{L^\infty}, \quad \|\gamma\|_{L_v^2} \lesssim \|u\|_{L_x^2},\quad \| |\nabla_v|^\beta \gamma\|_{L_v^2} \lesssim\| |J|^{\beta} u\|_{L_x^2}.
            \end{equation}
            We also have the physical space bounds
            \[
                |u(t, 2vt) - t^{-\frac{d}{2}} e^{i\frac{|x|^2}{4t}} \gamma(t, v)| \lesssim t^{-\frac{d}{2}-\frac{1}{10}} \| |J|^{\beta} u\|_{L_x^2}
            \]  
            and the Fourier space bounds 
            \[
                |\hat{u}(t, \xi) - e^{-it|\xi|^2} \gamma(t, \xi)| \lesssim t^{-\frac{d}{4}-\frac{1}{10}} \| |J|^{\beta}u\|_{L_x^2}.
            \]
        \end{lemma}
        \begin{proof}
            Let $w = e^{-i|x|^2/4t}u$. Then we can express $\gamma$ in terms of $w$ as a convolution with respect to the variable $v$: 
            \begin{equation}\label{E:gammaconvrep}
                t^{-\frac{d}{2}}\gamma(t, v) = w(t, 2vt) \ast_v 2^d t^{\frac{d}{2}} \vartheta(2t^\frac{1}{2}v),
            \end{equation}
            noting that $v \mapsto 2^d t^{\frac{d}{2}}\vartheta((2t)^{\frac{1}{2}}v)$ has integral 1. By Young's inequality, then, we have the immediate convolution bounds 
            \begin{align*}
                \|\gamma(t, v)\|_{L^\infty} &\lesssim t^{\frac{d}{2}}\|w(t, 2vt)\|_{L^\infty} = t^{\frac{d}{2}}\|u\|_{L^\infty}, \\
                \|\gamma(t, v)\|_{L_v^2} &\lesssim t^{\frac{d}{2}}\|w(t, 2vt)\|_{L_v^2} \sim \|u\|_{L_x^2}.
            \end{align*}
            A straightforward estimate via Young's inequality and the commutativity of Fourier multipliers yields the third bound, \textit{viz}
            \begin{align*}
                \| |\nabla_v|^{\beta} \gamma(t, v)\|_{L_v^2} &\lesssim \| |\nabla_v|^{\beta} w(t, 2vt)\|_{L_v^2} \\
                &\lesssim (2t)^{\beta-d} \left\|\int_{\R^d} e^{i\frac{\eta}{2t}\cdot x}\left|\frac{\eta}{2t}\right|^\beta \tilde{w}\left(t,\frac{\eta}{2t}\right)\d\eta \right\|_{L_x^2} \\
                &\lesssim (2t)^{\beta} \left\|\int_{\R^d} e^{i \xi \cdot x} |\xi|^{\beta} \hat{w}(t, \xi) \dxi \right\|_{L_x^2} \\
                &\lesssim  \| |J|^{\beta}u\|_{L_x^2},  
            \end{align*}
            where to get from the second line to the third we used the fact that $v = \tfrac{x}{2t}$ implies that their Fourier dual variables satisfy $\eta = 2t\xi$. To get rid of the $(2t)^\beta$, we used the definition \eqref{E:Jtdefn1} to absorb it back into the $|\nabla_x|^{\beta}$. \par
            For the physical space bounds, we need to compare $u(t, 2vt)$ with $\gamma(t, v)$. To do this, note that 
            \begin{equation}
                |u(t, 2vt) - t^{-\frac{d}{2}}e^{i\frac{|x|^2}{4t}}\gamma(t, v)| = |w(t, 2vt) - t^{-\frac{d}{2}}\gamma(t, v)|. 
            \end{equation}
            We can then use the representation \eqref{E:gammaconvrep} to write 
            \begin{equation}
                |w(t, 2vt) -t^{-\frac{d}{2}}\gamma(t, v)| = \left| \int_{\R^d} w(t, 2(v-z)t) (2t^{\frac{1}{2}})^d\vartheta((2t)^{\frac{1}{2}}z) \d z - w(t, 2vt)\right|.
            \end{equation}
            Using that $\vartheta(z)$ has unit integral, we can again rewrite the above equation as 
            \begin{equation}\label{E:differencerepn}
                |w(t, 2vt) -t^{-\frac{d}{2}}\gamma(t, v)| = \left| \int_{\R^d} [w(t, 2(v-z)t) - w(t, 2vt)](2t^{\frac{1}{2}})^d\vartheta((2t)^{\frac{1}{2}}z) \d z\right|.
            \end{equation}
            Then using homogeneous Sobolev embedding, we find (using that $\tfrac{d}{2}< \beta < 1+\tfrac{d}{2}$), 
            \begin{equation}
                |w(t, 2t(v-z)) - w(t, 2vt)| \lesssim |z|^{\beta - \frac{d}{2}} \| |\nabla_v|^{\beta} w(t, 2vt)\|_{L_v^2}.
            \end{equation}
            This allows us to continue the estimate of \eqref{E:differencerepn} by 
            \begin{align*}
                |w(t, 2vt) - t^{-\frac{d}{2}}\gamma(t, v)| &\lesssim \int_{\R^d} |z|^{\beta - \frac{d}{2}} \| |\nabla_v|^{\beta} w(t, 2vt)\|_{L_v^2} (2t^{\frac{1}{2}})^d |\vartheta((2t)^{\frac{1}{2}}z)| \d z \\
                &\lesssim \| |\nabla_v|^\beta w(t, 2vt)\|_{L_v^2} (2t^{\frac{1}{2}})^{-(\beta-\frac{d}{2})} \\
                &\lesssim t^{-\frac{\beta}{2}-\frac{d}{4}} \| |J|^{\beta}u\|_{L_x^2}. 
            \end{align*}
             
            To obtain the frequency estimate, we recall the definition \eqref{E:gammafreqrepn} and that $\tilde{\vartheta}$ has unit integral to write 
            \begin{equation}
                |\hat{u}(t, \xi) - e^{-it |\xi|^2} \gamma(t, \xi)| \leq \int_{\R^d} |  e^{it |\xi-\eta|^2}\hat{u}(\xi - \eta) - e^{it|\xi|^2}\hat{u}(\xi)| (2t)^{\frac{d}{2}} |\tilde{\vartheta}((2t)^{\frac{1}{2}}\eta)| \d\eta.
            \end{equation}
            Applying the same Sobolev embedding argument from earlier, we can bound the difference in the integral above by 
            \begin{equation}
                |e^{it|\xi - \eta|^2}\hat{u}(\xi - \eta) - e^{it|\xi|^2} \hat{u}(\xi)| \lesssim |\eta|^{\beta - \frac{d}{2}} \| |\nabla_\xi|^\beta (e^{it |\xi|^2} \hat{u})\|_{L_\xi^2}.
            \end{equation}
            Substituting this into the integral and applying the same argument as for the physical-space case yields the estimate 
            \begin{equation}
                |\hat{u}(t, \xi) - e^{-it|\xi|^2} \gamma(t, \xi)| \lesssim t^{\frac{d}{4}- \frac{\beta}{2}} \| |J|^{\beta}u\|_{L_x^2},
            \end{equation}
            completing the proof.
        \end{proof}
        
        \begin{lemma}
            Let $u(t, x)$ be a solution to \eqref{E:VNLS}. Setting $\beta = \frac{d}{2}+\frac{1}{10}$, we have
            \begin{equation}\label{E:gammaODE}
                \partial_t \gamma(t, v) = \frac{i}{2t} (|\cdot|^{-1} \ast |\gamma(t, v)|^2)\gamma(t, v) + \cR(t, v),
            \end{equation}
            where the remainder $\cR(t,v)$ satisfies the estimate 
            \begin{multline}\label{E:remainderestimate}
                \|\cR(t, v)\|_{L^\infty} \lesssim t^{-\frac{3}{2}}\|\jbrak{J}^\beta u\|_{L^2}+ {t^{-\frac{d}{2}+\frac{1}{2}-\frac{1}{20d}}\|u\|_{L^\infty}^{1+\frac{1}{d}} \|u\|_{L^2}^{\frac{2(d-1)}{d}} \| |J|^\beta u\|_{L^2}^\frac{1}{d}} \\+t^{\frac{d}{2}-\frac{1}{2}-\frac{1}{20d}}\|u\|_{L^\infty}^{1+\frac{1}{d}}\|u\|_{L^2}^{\frac{2(d-1)}{d}} \| |J|^\beta u\|_{L^2}^{\frac{1}{d}}.
            \end{multline}
        \end{lemma}    
        \begin{proof}
            Begin by taking the time derivative of $\gamma(t, v)$. By definition, this yields 
            \begin{align*}
                \partial_t \gamma(t, v) &= \int_{\R^d} \partial_t u \bbar{\Psi_v} + u \bbar{\partial_t \Psi_v} \dx \\
                &= \int_{\R^d} i \left[\Delta u - (|\cdot|^{-1} \ast |u|^2) u\right] \bbar{\Psi_v} + u \bbar{\partial_t \Psi_v} \dx \\
                &= i\int_{\R^d} u \bbar{\left[i \partial_t + \Delta\right]\Psi_v} \dx - i \int_{\R^d} (|\cdot|^{-1} \ast |u|^2)u \bbar{\Psi_v} \dx,
            \end{align*}
            where the integrations by parts to move the Laplace operator from the $u$ to the wavepacket $\Psi_v$ are justified using the decay of $\Psi_v$. Next, we use the computation \eqref{E:approxsoln} to write 
            \begin{equation}
                i \int_{\R^d} u \bbar{\left[i\partial_t + \Delta\right]\Psi_v} \dx = \frac{i}{2t} \int_{\R^d} u \M(-t) \nabla \cdot \left\{-i(x-2vt)\bbar{\vartheta} + 2t^{\frac{1}{2}}\bbar{\nabla\vartheta}\right\}\dx.
            \end{equation}
            We can then integrate by parts to rewrite the integral as 
            \begin{multline*}
                \frac{i}{2t} \int_{\R^d} u \M(-t) \nabla \cdot \left\{-it^{-\frac{1}{2}}(x-2vt)\bbar{\vartheta} + 2t^{\frac{1}{2}}\bbar{\nabla\vartheta}\right\}\dx \\= -\frac{i}{2t}\int_{\R^d} \nabla (\M(-t) u) \cdot \left(-i (x-2vt)\bbar{\vartheta} + 2t^{\frac{1}{2}} \bbar{\nabla \vartheta} \right)\dx.
            \end{multline*}
            We then use the definition of $J$ to rewrite this once more as 
                \begin{multline}
                    -\frac{i}{2t}\int_{\R^d} \nabla (e^{-i\frac{|x|^2}{4t}} u) \cdot \left(-i (x-2vt)\bbar{\vartheta} + 2t^{\frac{1}{2}} \bbar{\nabla \vartheta} \right)\dx \\= -\frac{1}{(2t)^2} \int_{\R^d} e^{-i\frac{|x|^2}{4t}} Ju \cdot \left(i (x-2vt)\bbar{\vartheta} -2t^\frac{1}{2}\bbar{\nabla \vartheta}\right) \dx.
                \end{multline}
            Combining this with the second term above, we can rewrite things in the following fashion: 
            \begin{align}
                \partial_t \gamma(t, v) &= -\frac{1}{(2t)^2}\int_{\R^d} \M(-t) Ju \cdot \{-i (x-2vt) \bbar{\vartheta} +2t^{\frac{1}{2}}\bbar{\nabla \vartheta}\} \dx \\
                &-i \int_{\R^d} u\bbar{\Psi_v} (|\cdot|^{-1} \ast (|u|^2 - |u(t, 2vt)|^2)) \dx \\
                &+ \begin{multlined}[t]
                    i \gamma |\cdot|^{-1} \ast |u(t, vt)|^2 - i \frac{1}{2t} (|\cdot|^{-1} \ast |\gamma(t, v)|^2 )\gamma \\+ i \frac{1}{2t}(|\cdot|^{-1} \ast |\gamma(t, v)|^2)\gamma
                \end{multlined} \\
                &\defe \frac{i}{2t}(|\cdot|^{-1} \ast |\gamma(t, v)|^2 )\gamma + \mathcal{R}_1 + \mathcal{R}_2 + \mathcal{R}_3,
            \end{align}
            where we define 
            \begin{equation}
                \cR_1 \defe -\frac{1}{(2t)^2}\int_{\R^d} \M(-t) Ju \cdot \{-i (x-2vt) \bbar{\vartheta} +2t^{\frac{1}{2}}\bbar{\nabla \vartheta}\} \dx,
            \end{equation}
            \begin{equation}
               \cR_2 \defe -i \int_{\R^d} u\bbar{\Psi_v} (|\cdot|^{-1} \ast (|u|^2 - |u(t, 2vt)|^2)) \dx, 
            \end{equation} 
            and 
            \begin{equation}
                \cR_3 \defe -\gamma(t, v)[|\cdot|^{-1} \ast (|u(t, 2vt)|^2 - \frac{1}{2t} |\gamma(t, v)|^2)],
            \end{equation}
            and where here and throughout we abide by the convention in \Cref{R:convolution}.
            We can now proceed to estimate each of the terms $\|\cR_j\|_{L^\infty}$. For $\cR_1$, we can change variables $x = 2tz$ to rewrite $\cR_1$ as a convolution: 
            \begin{align}
                \cR_1 &=
                \begin{multlined}[t]
                    -(2t)^{d-2}\int_{\R^d} \tilde{\M}(-t) Ju(2tz) \\\times \{-2it(z-v)\bbar{\vartheta(2t^{\frac{1}{2}}(z-v))} + 2t^{\frac{1}{2}} \bbar{\nabla \vartheta(2t^{\frac{1}{2}}(z-v))}\}\dz 
                \end{multlined}\\
                &= -\frac{1}{(2t)^2} \tilde{\M}(-t) (Ju)(2tz) \ast_v \left\{-2itz \vartheta(2t^{\frac{1}{2}}z) + 2t^{\frac{1}{2}} \bbar{\nabla \vartheta(2t^{\frac{1}{2}}z)}\right\}.
            \end{align}
            Here we write $\tilde{\M}(-t) = e^{-it|z|^2}$. Using Young's convolution inequality, we can place the first part of the convolution in $L^\infty$ and the second in $L^1$ to obtain 
            \begin{align}
                \|\cR_1\|_{L^\infty} &\lesssim t^{d-2} \|Ju\|_{L^\infty} t^{\frac{1}{2}-\frac d2} \\
                &\lesssim t^{\frac{d}{2}-\frac{3}{2}} \|e^{it\Delta} x e^{-it\Delta}u\|_{L^\infty}\\
                &\lesssim t^{- \frac{3}{2}}\|x e^{-it\Delta}u \|_{L^1} \\
                &\lesssim t^{-\frac{3}{2}}\|\jbrak{x}^\beta e^{-it\Delta}u \|_{L^2} \\
                &\lesssim t^{-\frac{3}{2}}\{\|u\|_{L^2} + \||J|^\beta u\|_{L^2}\}\\
                &\lesssim t^{-\frac{3}{2}}\|\jbrak{J}^\beta u\|_{L^2}
            \end{align}
            which is an acceptable estimate. 
            For $\cR_2$, we have 
            \begin{equation}
                \cR_2 = -i\int_{\R^d} u \bbar{\Psi}_v \left[|\cdot|^{-1} \ast(|u|^2(x) - |u|^2(2tv))\right] \dx.
            \end{equation}
            We thus have (using the same notation for $w(t,x)$ as above)
            \begin{align}
                |\cR_2| &\lesssim \|u\|_{L^\infty} \int_{\R^d} \bbar{\vartheta}\left(\frac{x-2tv}{t^{\frac{1}{2}}}\right)\left[|\cdot|^{-1} \ast (|w|^2(x) - |w|^2(2tv))\right] \dx \\
                &\lesssim \|u\|_{L^\infty} t^{\frac{d}{2}}\int_{\R^d} 2^d t^{\frac{d}{2}}\bbar{\vartheta}(2t^{\frac{1}{2}}z)\left[ |\cdot|^{-1} \ast (|w|^2(2t(v-z)) - |w|^2(2tv)) \right] \dz \\
                &\lesssim \|u\|_{L^\infty}t^\frac{d}{2}\int_{\R^d} 2^d t^{\frac{d}{2}} \bbar{\vartheta}(2t^{\frac{1}{2}}z) \left\| |\cdot|^{-1} \ast (|w|^2(2t(v-z)) - |w|^2(2tv))\right\|_{L_v^\infty} \dz. \label{E:R2Linftyest}
            \end{align}
            Now we need to estimate the convolution in $L^\infty$. To do this, we use \Cref{L:interplemma}, which yields
            \begin{equation*}
                \||\cdot|^{-1} \ast (|w|^2(2t(v-z)) - |w|^2(2tv))\|_{L_v^\infty}\lesssim \| |w|^2(2t(v-z)) - |w|^2(2tv)\|_{L^{\frac{d}{d-1},1}}.
            \end{equation*}
            The right-hand side of this equation is bounded by 
            \begin{equation*}
            | |w|^2(2t(v-z)) - |w|^2(2tv)\|_{L^{\frac{d}{d-1},1}} \lesssim \begin{multlined}[t]
                \|w(2t(v-z)) - w(2tv)\|_{L_v^2}^{\frac{d-1}{d}} \\ \times \|w(2t(v-z)) - w(2tv)\|_{L_v^\infty}^{\frac{1}{d}} \\
                   \times \| w(2t(v-z)) + w(2tv)\|_{L_v^2}^{\frac{d-1}{d}} \\ \times \|w(2t(v-z)) + w(2tv)\|_{L_v^\infty}^{\frac{1}{d}}.
            \end{multlined}
            \end{equation*}
            We bound each term individually. We crudely control the $L^2$ norms by 
            \begin{equation*}
                \|w(2t(v-z)) \pm w(2tv)\|_{L_v^2}^{\frac{d-1}{d}} \lesssim t^{-\frac{d-1}{2}}\|u\|_{L^2}^{\frac{d-1}{d}}
            \end{equation*}
            the $L^\infty$ norm of the sum by 
            \begin{equation*}
                \|w(2t(v-z)) + w(2tv) \|_{L_v^\infty}^{\frac{1}{d}} \lesssim \|u\|_{L^\infty}^{\frac{1}{d}},
            \end{equation*}
            and we use Sobolev embedding to control 
            \begin{equation*}
                \|w(2t(v-z)) - w(2tv)\|_{L_v^\infty}^{\frac{1}{d}} \lesssim |z|^{\frac{\beta}{d}-\frac{1}{2}}\||J|^\beta u\|_{L^2}^{\frac{1}{d}} t^{-\frac{1}{2}}.
            \end{equation*}
            Putting this all together, we see that we can estimate the right-hand side of  \eqref{E:R2Linftyest} by 
            \begin{align*}
                \mathrm{RHS}\eqref{E:R2Linftyest}&\lesssim \|u\|_{L^\infty}^{1+\frac{1}{d}} \|u\|_{L^2}^{\frac{2(d-1)}{d}} \| |J|^\beta u\|_{L^2}^\frac{1}{d}t^{\frac{d}{2}-d+1-\frac{1}{2}} \int_{\R^d} 2^d t^{\frac{d}{2}}\bbar{\vartheta}(2t^{\frac{1}{2}}z) |z|^{\frac{\beta}{d}-\frac{1}{2}} \dz \\
                &\lesssim \|u\|_{L^\infty}^{1+\frac{1}{d}} \|u\|_{L^2}^{\frac{2(d-1)}{d}} \| |J|^\beta u\|_{L^2}^\frac{1}{d} t^{-\frac{d}{2}+\frac{3}{4}-\frac{\beta}{2d}}
            \end{align*}
            which is an acceptable estimate.\par
            Finally, we need an estimate on $\cR_3$. We have 
            \begin{equation}
                \cR_3 = -\gamma(t, v) \left[|\cdot|^{-1} \ast \left(|u(t, 2vt)|^2 - \frac{1}{2t}|\gamma(t, v)|^2\right)\right].
            \end{equation}
            Now notice that we can write 
            \begin{align}
                |\cdot|^{-1} \ast |u(t, 2vt)|^2 &= \int_{\R^d} \frac{|u(t, y)|^2}{|2vt- y|} \dy \\
                &= (2t)^{d-1} \int_{\R^d} \frac{|u(t, 2tz)|^2}{|v - z|} \dz.
            \end{align}
            Using this, $\cR_3$ can be rewritten as 
            \begin{equation}
                \cR_3 = -(2t)^{d-1}\gamma(t, v) \int_{\R^d} \frac{|u(2tz)|^2 - (2t)^{-d} |\gamma(t,z)|^2}{|v-z|}\dz.
            \end{equation}
            Now using the $L^\infty$ endpoint of the Hardy-Littlewood-Sobolev inequality \cite[Theorem 2.6]{oneilConvolutionOperatorsSpaces1963}, we can estimate the integral by 
            \begin{align}
                \left|\int_{\R^d} \frac{|w(2tz)|^2 - (2t)^{-d}|\gamma(t, z)|^2}{|v-z|}\dz\right| &\lesssim \left\| |w(2tz)|^2 - (2t)^{-d}|\gamma(t, z)|^2\right\| _{L^{\frac{d}{d-1},1}} \\
                &\lesssim \begin{multlined}[t]
                    \left\|w(2tz) - (2t)^{-\frac{d}{2}}\gamma(t, z)\right\|_{L^{\frac{2d}{d-1},2}} \\\times\left\|w(2tz) + (2t)^{-\frac{d}{2}}\gamma(t, z)\right\|_{L^{\frac{2d}{d-1}, 2}}.
                \end{multlined}
            \end{align}
            We now need to estimate each of the Lorentz norms above. By \Cref{L:interplemma}, we find that 
            \begin{align}
                &\begin{multlined}
                    \left\| w(2tz) - (2t)^{-\frac{d}{2}}\gamma(t, z)\right\|_{L^{\frac{2d}{d-1}, 2}} \\\lesssim \|w(2tz) - (2t)^{-\frac{d}{2}}\gamma(t, z)\|_{L^2}^{\frac{d-1}{d}}\left\|w(2tz) - (2t)^{-\frac{d}{2}}\gamma(t, z)\right\|_{L^\infty}^{\frac{1}{d}}\label{E:minusterm}
                \end{multlined}
                \\
                &\begin{multlined}
                    \left\|w(2tz) + (2t)^{-\frac{d}{2}}\gamma(t, z)\right\|_{L^{\frac{2d}{d-1}, 2}} \\\lesssim \left\|w(2tz) + (2t)^{-\frac{d}{2}}\gamma(t, z)\right\|_{L^2}^{\frac{d-1}{d}}\left\|w(2tz) + (2t)^{-\frac{d}{2}}\gamma(t, z)\right\|_{L^\infty}^{\frac{1}{d}}.\label{E:plusterm}
                \end{multlined}
            \end{align}
            We now estimate each of the norms on the right-hand side individually. \par
            For the terms in \eqref{E:plusterm}, we use \Cref{L:gammabounds} and the triangle inequality: 
            \begin{align}
                &\|w(2tz)+(2t)^{-\frac{d}{2}}\gamma(t, z)\|_{L^2}^{\frac{d-1}{d}} \lesssim t^{-\frac{d-1}{2}}\|u\|_{L^2}^{\frac{d-1}{d}}, \\
                &\|w(2tz) + (2t)^{-\frac{d}{2}}\gamma(t, z)\|_{L^\infty}^{\frac{1}{d}} \lesssim \|u\|_{L^\infty}^{\frac{1}{d}}.
            \end{align}
            For the terms in \eqref{E:minusterm}, we crudely estimate the difference in $L^2$ by the sum, and apply the analysis from \eqref{E:plusterm}. For the $L^\infty$ term, though, we use the physical space estimate from \Cref{L:gammabounds}. This nets us the pair of estimates 
            \begin{align}
                &\|w(2tz) - (2t)^{-\frac{d}{2}}\gamma(t, z)\|_{L^2}^{\frac{d-1}{d}} \lesssim t^{-\frac{d-1}{2}}\|u\|_{L^2}^{\frac{d-1}{d}}, \\
                &\|w(2tz) - (2t)^{-\frac{d}{2}}\gamma(t, z)\|_{L^\infty}^{\frac{1}{d}} \lesssim t^{-\frac{1}{d}(\frac{\beta}{2}+\frac{d}{4})}\| |J|^\beta u\|_{L^2}^{\frac{1}{d}}.
            \end{align}
            Note that the physical space estimate can be reproven with an extra factor of $2$; all that changes is that in \eqref{E:differencerepn} we have a $2^{d/2}$ that we can freely replace with $2^d$, using the fact that $d > 0$.
            
            Putting all this information together, we establish the final estimate 
            \begin{equation}
                    \|\cR_3\|_{L^\infty} \lesssim t^{\frac{d}{2}-\frac{1}{d}(\frac{\beta}{2}+\frac{d}{4})}\|u\|_{L^2}^{\frac{2(d-1)}{d}} \|u\|_{L^\infty}^{1+\frac{1}{d}}\| |J|^\beta u\|_{L^2}^{\frac{1}{d}},
            \end{equation}
            which is acceptable. 
        \end{proof}
        A similar argument gives the following lemma in the case of \eqref{E:SBP}:
        \begin{lemma}\label{L:SBPgammaode}
            Let $u(t, x)$ be a solution to \eqref{E:SBP}. Then we have the following ODE for the associated wavepacket $\gamma$:
            \begin{equation}\label{E:SBPode}
                \partial_t \gamma(t, v) = -\frac{i}{2t}\left\{ (\K \ast |\gamma|^2)\gamma - |\gamma|^{\frac2d}\gamma\right\} + \cR(t, v), 
            \end{equation}
            where the remainder $\cR$ satisfies the estimate
            \begin{equation*}
                \|\cR(t, v)\|_{L^\infty} \lesssim\begin{multlined}[t]
                t^{-\frac{3}{2}}\|\jbrak{J}^\beta u\|_{L^2}+ {t^{-\frac{d}{2}+\frac{1}{2}-\frac{1}{20d}}\|u\|_{L^\infty}^{1+\frac{1}{d}} \|u\|_{L^2}^{\frac{2(d-1)}{d}} \| |J|^\beta u\|_{L^2}^\frac{1}{d}} \\+t^{\frac{d}{2}-\frac{1}{2}-\frac{1}{20d}}\|u\|_{L^\infty}^{1+\frac{1}{d}}\|u\|_{L^2}^{\frac{2(d-1)}{d}} \| |J|^\beta u\|_{L^2}^{\frac{1}{d}} + t^{\frac{d}{2} - \frac{1}{2}- \frac{1}{10d}}\|u\|_{L^\infty}^{1+\frac{1}{d}}\||J|^\beta u\|_{L^2}^{\frac{1}{d}}.
                \end{multlined} 
            \end{equation*}
        \end{lemma} 
        \begin{proof}
            This follows from the previous lemma and the pointwise bound $|\K(x)| \leq |x|^{-1}$, with additional terms coming from the power-type nonlinearity $|u|^{\frac{2}{d}}u$. These take the form
            \begin{align*}
                \cR_A(t, v) &= i \int_{\R^d} u\bbar{\Psi}_v (|u|^{\frac{2}{d}} - |u(2tv)|^{\frac{2}{d}})\dx \\
                \cR_B(t, v) &=  i \gamma \left(|u(2tv)|^{\frac{2}{d}} - \frac{i}{2t}|\gamma|^\frac{2}{d}\right).
            \end{align*}
            For the first term, we estimate in $L_v^\infty$ by 
            \begin{equation*}
                |\cR_A(t, v)| \lesssim \|u\|_{L^\infty}^{1+\frac{1}{d}} \int_{\R^d} \bbar{\vartheta}\left(\frac{x-2tv}{t^\frac{1}{2}}\right)(|w|^\frac{1}{d} - |w(2tv)|^\frac{1}{d}) \dx.
            \end{equation*}
            Using the change of variables $x = 2tz$, convexity of the map $t \mapsto t^{\frac{1}{d}}$ for $t \geq 0$ and $d \geq 1$, and the Sobolev embedding estimate from above, we get the bound 
            \begin{equation*}
                \|\cR_A(t, v)\|_{L_v^\infty} \lesssim t^{\frac{d}{2} - \frac{1}{2}- \frac{1}{10d}}\|u\|_{L^\infty}^{1+\frac{1}{d}}\||J|^\beta u\|_{L^2}^{\frac{1}{d}},
            \end{equation*}
            which is acceptable. For the term $\cR_B$, we have 
            \begin{equation*}
                |\cR_B(t, v)| \lesssim t^{\frac{d}{2}}\|u\|_{L^\infty}^{\frac{1}{d}}\left||u(2tv)|^\frac{1}{d}- t^{-\frac{1}{2}}|\gamma|^\frac{1}{d}\right|,
            \end{equation*}
            which by convexity and \Cref{L:gammabounds} implies that we have the bound 
            \begin{equation*}
                \|\cR_B(t, v)|\lesssim t^{\frac{d}{2} - \frac{1}{2}-\frac{1}{10d}} \|u\|_{L^\infty}^{1+\frac{1}{d}}\||J|^\beta u\|_{L^2}^\frac{1}{d},
            \end{equation*}
            which is also acceptable. 
        \end{proof}
    Now that we have derived an approximate ODE for the wavepacket $\gamma$ and bounds on the remainder terms, we want to prove the global bounds claimed in the statement of the main theorems. This statement is furnished by the following lemma: 
    \begin{lemma}
        Suppose that $0 < \eps \ll 1$ and suppose $\|u_0\|_{H^{0, \beta}} = \eps$. Let $u(t, x)$ be the corresponding global solution to \eqref{E:VNLS} furnished by \Cref{T:GWP}. Then the following bounds hold globally in time: 
        \begin{align*}
            \|u\|_{L^\infty} &\lesssim \eps |t|^{-\frac{d}{2}} \\
            \| \jbrak{J}^\beta u\|_{L^2} &\lesssim 2\eps \jbrak{t}^{\tilde{C}\eps^2},
        \end{align*}
        for some constant $\tilde{C} \ll \eps^{-\frac{2}{d}}$. We have a similar bound for \eqref{E:SBP}, but the bound on $\| \jbrak{J}^\beta u\|_{L^2}$ is replaced by 
        \begin{equation*}
            \| \jbrak{J}^\beta u\|_{L^2} \lesssim 2\eps \jbrak{t}^{\tilde{C}\eps^2 + (C\eps)^{\frac{1}{d}}},
        \end{equation*}
        for a constant $C \ll \eps^{-1}$ and $\tilde{C}$ as above.
    \end{lemma}
    \begin{proof}
        We will use a bootstrap argument, beginning under the assumption that 
        \begin{equation*}
            \|u(t)\|_{L^\infty} \leq C\eps|t|^{-\frac{d}{2}},
        \end{equation*}
        for some $C \ll \frac{1}{\eps}$.
        To prove the bound on $\| \jbrak{J}^\beta u\|_{L^2}$, we first note that \Cref{T:GWP} implies that 
        \begin{equation*}
            \| |J|^\beta u(1)\|_{L^2} < 2\eps.
        \end{equation*}
        For times $t > 1$, we use the Duhamel formula, started from time $t = 1$. This gives 
        \begin{equation*}
            u(t) = e^{i(t-1)\Delta}u(1) - i \int_1^t e^{i(t-s)\Delta}[(|\cdot|^{-1} \ast |u|^2)u]\ds,
        \end{equation*}
        in the Hartree case, or 
        \begin{equation*}
            u(t) = e^{i(t-1)\Delta}u(1) - i \int_1^t e^{i(t-s)\Delta}[(\K \ast |u|^2)u - |u|^\frac{2}{d}u] \ds 
        \end{equation*}
        in the case of Schr\"odinger-Bopp-Podolsky. In either case, applying $|J|^\beta$, estimating in $L^2$ and applying the triangle inequality, we see that we need to understand the following $3$ terms: 
        \begin{align}
            &\||\nabla|^\beta [(|\cdot|^{-1} \ast |u|^2)u]\|_{L^2}, \label{E:hartreederivs} \\
            & \||\nabla|^\beta [(\K \ast |u|^2)u]\|_{L^2},\label{E:Kderivs} \text{ and } \\
            &\| |\nabla|^\beta |u|^\frac{2}{d}u\|_{L^2} \label{E:powderivs}.
        \end{align}
        Note that the corresponding result for $|J|^\beta$ follows in the same way as in the persistence of regularity argument in \Cref{T:GWP}.

        For the first term, we see that by the fractional product rule, we have an estimate of the form
        \begin{equation*}
            \eqref{E:hartreederivs} \lesssim \| |\cdot|^{-1} \ast |\nabla|^\beta|u|^2\|_{L^p} \|u\|_{L^q} + \| |\cdot|^{-1} \ast |u|^2\|_
            {L^\infty} \| |\nabla|^\beta u\|_{L^2},
        \end{equation*}
        where $\frac{1}{p} + \frac{1}{q} = \frac{1}{2}$. By the endpoint Hardy-Littlewood-Sobolev inequality from above and \Cref{L:interplemma}, we have 
        \begin{equation*}
            \| |\cdot|^{-1} \ast |u|^2\|_{L^\infty} \|u\|_{L^2} \lesssim \|u\|_{L^2}^{\frac{2(d-1)}{d}} \|u\|_{L^\infty}^\frac{2}{d} \||\nabla|^\beta u\|_{L^2}.
        \end{equation*}
        For the first term, by applying the ordinary Hardy-Littlewood-Sobolev inequality and the fractional product rule, we conclude that 
        \begin{equation*}
            \| |\cdot|^{-1} \ast |\nabla|^\beta|u|^2\|_{L^p} \lesssim \|u\|_{L^q} \| |\nabla|^\beta u\|_{L^2} \|u\|_{L^r}, 
        \end{equation*}
        where we have 
        \begin{align*}
            \frac{1}{p} + \frac{d-1}{d} &= \frac{1}{2} + \frac{1}{r}, \text{ and } \\
            \frac{1}{p} + \frac{1}{q} &= \frac{1}{2}. 
        \end{align*}
        These conditions imply that both $q, r$ are larger than $2$, so that we can write the $L^q$ and $L^r$ norms as interpolants of the $L^2$ and $L^\infty$ norms, concluding that 
        \begin{equation*}
            \| |\cdot|^{-1} \ast |\nabla|^\beta|u|^2\|_{L^p} \lesssim \|u\|_{L^2}^{\frac{2(d-1)}{d}} \|u\|_{L^\infty}^\frac{2}{d} \||\nabla|^\beta u\|_{L^2},
        \end{equation*}
        which is acceptable. In light of the pointwise bound $\K(x) \leq \frac{1}{|x|}$, the estimates above apply to yield the same estimate for $\| |\nabla|^\beta (\K \ast |u|^2) u\|_{L^2}$.

        Finally, for the term $\| |\nabla|^\beta |u|^\frac{2}{d}u\|_{L^2}$, we apply the fractional chain rule to conclude that 
        \begin{equation*}
            \| |\nabla|^\beta |u|^\frac{2}{d} u\|_{L^2} \lesssim \| |\nabla|^\beta u\|_{L^2} \|u\|_{L^\infty}^\frac{2}{d},
        \end{equation*}
        which is also acceptable. 

        In total, we have the estimates 
        \begin{equation*}
            \| |J|^\beta u\|_{L^2} \lesssim 2\eps + \int_1^t \|u\|_{L^2}^{\frac{2(d-1)}{d}} \|u\|_{L^\infty}^\frac{2}{d} \||J|^\beta u\|_{L^2} \ds,
        \end{equation*}
        in the case of \eqref{E:VNLS}, and 
        \begin{equation*}
            \| |J|^\beta u\|_{L^2} \lesssim 2\eps + \int_1^t \|u\|_{L^2}^{\frac{2(d-1)}{d}} \|u\|_{L^\infty}^\frac{2}{d} \||J|^\beta u\|_{L^2} + \| J|^\beta u\|_{L^2} \|u\|_{L^\infty}^\frac{2}{d}\ds,
        \end{equation*}
        for \eqref{E:SBP}. Under the assumptions on the initial data and the bootstrap hypotheses, we have 
        \begin{align*}
            \| |J|^\beta u\|_{L^2} &\lesssim 2\eps +  \int_1^t \eps^{\frac{2(d-1)}{d}}(C\eps)^{\frac{2}{d}} \frac{1}{s} \| |J|^\beta u\|_{L^2} \ds \qtq{for \eqref{E:VNLS}, and} \\
            \| |J|^\beta u\|_{L^2} &\lesssim 2\eps + \int_1^t (\eps^{\frac{2(d-1)}{d}}(C\eps)^{\frac{2}{d}} + (C\eps)^\frac{2}{d}) \frac{1}{s}\| |J|^\beta u\|_{L^2} \ds \qtq{for \eqref{E:SBP}.}
        \end{align*}
        In particular, by Gr\"onwall and mass conservation, we conclude that there exists $\tilde{C} \ll \eps^{-\frac{2}{d}}$ so that
        \begin{equation*}
            \| \jbrak{J}^\beta u\|_{L^2} \lesssim 2\eps \jbrak{t}^{\tilde{C}\eps^{2}} \qtq{for \eqref{E:VNLS},}
        \end{equation*}
        and 
        \begin{equation*}
            \| \jbrak{J}^\beta u\|_{L^2} \lesssim 2\eps \jbrak{t}^{\tilde{C}\eps^{2} + (C\eps)^{\frac{2}{d}}} \qtq{for \eqref{E:SBP}.}
        \end{equation*}

        It now remains to close the bootstrap and establish the claimed pointwise decay estimate. First, by the triangle inequality we have that 
        \begin{equation*}
            |u(t, 2vt)| \leq |t|^{-\frac{d}{2}}|\gamma(t,v)| + |e^{i\frac{|x|^2}{4t}}u(t, 2vt) - t^{-\frac{d}{2}}\gamma(t,v)|.
        \end{equation*}
        In light of \Cref{L:gammabounds}, the second term is controlled by 
        \begin{equation}\label{E:diffctrlbd}
            |e^{i\frac{|x|^2}{4t}}u(t, 2vt) - t^{-\frac{d}{2}}\gamma(t,v)| \lesssim t^{-\frac{d}{2}-\frac{1}{10}}\| |J|^\beta u \|_{L_x^2},
        \end{equation}
        and so combining this with the bounds for $\| \jbrak{J}^\beta u\|_{L^2}$ from the Gr\"onwall argument above, we conclude that 
        \begin{equation*}
            \eqref{E:diffctrlbd} \lesssim 2\eps|t|^{-\frac{d}{2}- \frac{1}{10d} + \tilde{C}\eps^2} \qtq{for \eqref{E:VNLS}},
        \end{equation*}
        and in the case of \eqref{E:SBP}
        \begin{equation*}
            \eqref{E:diffctrlbd} \lesssim 2\eps|t|^{-\frac{d}{2}- \frac{1}{10d} + \tilde{C}\eps^2 + (C\eps)^\frac{2}{d}}.
        \end{equation*}
        This term is acceptable; with our choice of sufficiently small $\eps$, we see that it is decaying faster than $t^{-\frac{d}{2}}$.
        All that remains is to estimate $|\gamma(t, v)|$. At this point we will prove the result for only \eqref{E:SBP}, as the argument for \eqref{E:VNLS} will work in the same way with only minor modifications.

        To this end, we introduce the integrating factor 
        \begin{equation*}
            B_{\rm{SBP}}(t) = \exp\left(-\int_{1}^{t}\frac{i}{2s}(\K \ast |\gamma(s, v)|^2 - |\gamma(s, v)|^\frac{2}{d}) \ds\right)
        \end{equation*}
        Noting that this function satisfies $|B_{\rm{SBP}}(t)| = 1$ for all $t$ and multiplying this function through \eqref{E:SBPode}, we see that after taking absolute values and applying the triangle inequality, we have
        \begin{equation*}
            |\gamma(t,v)| \lesssim |\gamma(1,v)| + \int_1^t |\cR(s,v)|\ds.
        \end{equation*}
        Using the bounds from \Cref{L:SBPgammaode} and inserting the bootstrap hypothesis for $\|u\|_{L^\infty}$, we can estimate $\cR$ by 
        \begin{equation*}
            \|\cR(s,v)\|_{L_v^\infty} \lesssim \begin{multlined}[t]
                \eps\jbrak{s}^{-\frac{3}{2}} + 2\eps \jbrak{s}^{-\frac{3}{2}+\tilde{C}\eps^2+ (C\eps)^{\frac{2}{d}}} + K\eps^3 \jbrak{s}^{-d-\frac{1}{20d}+\tilde{C}\eps^2+(C\eps)^\frac{2}{d}} \\ 
                + K\eps^3\jbrak{s}^{-1-\frac{1}{20d}+\tilde{C}\eps^2+(C\eps)^{\frac{2}{d}}} + K\eps^{1+\frac{2}{d}}\jbrak{s}^{-1-\frac{1}{10d}+\tilde{C}\eps^2+(C\eps)^{\frac{2}{d}}},
            \end{multlined}
        \end{equation*}
        where we set $K = C^{1+\frac{1}{d}}$ for notational brevity, and we recall that $\tilde{C} \ll \eps^{-\frac{2}{d}}$. In particular, examining the exponents in the powers of $\jbrak{s}$, we need to choose $\eps \ll 1$ so that the hierarchy 
        \begin{equation*}
            \tilde{C}\eps^2 \ll \eps^{2-\frac{2}{d}} \ll (C\eps)^{\frac{2}{d}} < \frac{1}{20d}
        \end{equation*}
        holds. Indeed, such a choice ensures that all the powers of $s$ are integrable in $t$, and this yields the estimate 
        \begin{equation*}
            |\gamma(t, v)| \lesssim 3\eps + K\eps^2 + K\eps^{1+\frac{2}{d}}.
        \end{equation*}
        Choosing $\eps$ sufficiently small and recalling the definition of $K$ from above, we can guarantee that 
        \begin{equation*}
            |\gamma(t,v)| \lesssim \eps.
        \end{equation*}
        By applying the triangle inequality in the form 
        \begin{equation*}
            |u(t,2vt)| \leq |t|^{-\frac{d}{2}}|\gamma(t, v)| + |e^{i\frac{|x|^2}{4t}}u(t,2vt) - t^{-\frac{d}{2}}\gamma(t,v)|,
        \end{equation*}
        inserting the estimates on both terms from above, and taking a supremum in $v$, we conclude the desired estimate for $u$.
    \end{proof}
    \section{Asymptotic Behavior of Solutions}\label{S:asymptotics}
        In the previous sections, we closed the bootstrap argument which established the sharp decay estimates for solutions to \eqref{E:VNLS} and \eqref{E:SBP}. Using these estimates, we will now prove the claimed asymptotic expansions from \Cref{T:maintheorem}.

        Recall the integrating factor 
        \begin{equation}
            B_{\rm{SBP}}(t) = \exp\left(-\int_1^t \frac{i}{2s}(\K \ast |\gamma(s,v)|^2 - |\gamma(s, v)|^{\frac{2}{d}})\ds\right)
        \end{equation}
        associated to \eqref{E:SBPode}. We have a similar integrating factor for \eqref{E:gammaODE}, \emph{viz}. 
        \begin{equation}
            B_{\rm{Har}}(t) = \exp\left(-\int_1^t \frac{i}{2s}(|\cdot|^{-1}\ast|\gamma(s,v)|^2)\ds\right).
        \end{equation}

        We proceed with the Hartree case first. Begin by setting $G_{\rm{Har}}(t) = B_{\rm{Har}}(t)\gamma(t,v)$. Then using \eqref{E:gammaODE} and the bootstrap argument from earlier, we deduce
        \begin{equation}
            \|\partial_t G_{\rm{Har}}(t)\|_{L_v^\infty} \lesssim \eps^3t^{-1-\frac{1}{20d}+\tilde{C}\eps^2}.
        \end{equation}
        By the fundamental theorem of calculus, there exists $\mathscr{W}_0 \in L^\infty$ so that
        \begin{equation}
            G_{\rm{Har}}(t) \to \mathscr{W}_0 \qtq{in} L^\infty \text{ as } t\to \infty.
        \end{equation}
        Taking absolute values in the definition of $G_{\rm{Har}}$ and using that $B_{\rm{Har}}$ has modulus $1$, it follows that $|\gamma(t, v)| \to |\mathscr{W}_0|$ in the same norm as $t \to \infty$. In particular, inserting this expansion into $B_{\rm{Har}}(t)$, we see that 
        \begin{equation*}
            B_{\rm{Har}}(t) = \exp(-i\log(t)(|\cdot|^{-1} \ast |\mathscr{W}_0|^2) - i \Phi(t)),
        \end{equation*}
        where again the function $\Phi(t)$ has some limit $\Phi_\infty$ in $L^\infty$ as $t \to \infty$. Splitting off this limiting behavior by setting $\mathscr{W} = e^{-i\Phi_\infty}\mathscr{W}_0$, it immediately follows that we have the expansion 
        \begin{equation*}
            \gamma(t,v) = e^{-i\log(t)(|\cdot|^{-1} \ast |\mathscr{Q}(v)|^2)}\mathscr{W}(v) + \mathcal{O}(t^{-\frac{1}{10d}}) \qtq{as} t \to \infty.
        \end{equation*}
        By the same replacement/triangle inequality argument as in the proof of the bootstrap argument, this completes the proof of the asymptotic expansion. 

        A virtually identical argument gives the result for \eqref{E:SBP}, upon replacing the integrating factor with the corresponding one for \eqref{E:SBPode} and running the same argument with the fundamental theorem of calculus.
    \appendix
        \section{Interpolation and Lorentz Spaces} \label{AS:InterpAppendix}
            In this appendix, we will give the details of the claim in \Cref{L:interplemma}. To begin, \cite[Theorem 5.3.1]{berghInterpolationSpacesIntroduction1976} is as follows:
            \begin{theorem*}
                Suppose that $0 < p_0, p_1, q_0, q_1 \leq \infty$ and write 
                \begin{equation*}
                    \frac{1}{p} = \frac{1-\theta}{p_0} + \frac{\theta}{p_1} \qtq{where} 0 < \theta < 1.
                \end{equation*}
                Then if $p_0 \neq p_1$, we have 
                \begin{equation*}
                    [L^{p_0, q_0}, L^{p_1, q_1}]_{\theta, q} = L^{p, q}.
                \end{equation*}
                If $p_0 = p_1 = p$, this result holds provided that 
                \begin{equation*}
                    \frac{1}{q} = \frac{1-\theta}{q_0} + \frac{\theta}{q_1}.
                \end{equation*}
            \end{theorem*}
            Applying this theorem with $(p_0, q_0) = (2, 2)$, $(p_1, q_1) = (\infty, \infty)$ and $(p, q) = \left( \frac{2d}{d-1}, 2\right)$ yields the first claim of \Cref{L:interplemma}, as the resulting value for $\theta$ is $\theta = \frac{1}{d} \in (0, 1)$, so that the above Theorem applies. 

            To conclude the remainder of \Cref{L:interplemma} involving interpolation of norms, we note that by definition the space $L^{p, q}$ in the Theorem above is defined by real interpolation. Using \cite[Theorem 3.5.1]{berghInterpolationSpacesIntroduction1976}, we see that $L^{\frac{2d}{d-1}, 2}$ (in the notation of \cite{berghInterpolationSpacesIntroduction1976}) is a space of class $\mathscr{C}\left(\frac{1}{d}, (L^2, L^\infty)\right)$. Using \cite[Item (b), p.49]{berghInterpolationSpacesIntroduction1976} and the fact that being class $\mathscr{C}\left(\frac{1}{d}, (L^2, L^\infty)\right)$ implies being class $\mathscr{C}_J\left(\frac{1}{d}, (L^2, L^\infty)\right)$, we conclude that 
            \begin{equation*}
                \|u\|_{L^{\frac{2d}{d-1}, 2}} \lesssim \|u\|_{L^2}^{\frac{d-1}{d}}\|u\|_{L^\infty}^{\frac{1}{d}},
            \end{equation*}
            which is the claim.
        \nocite{*}
    \bibliography{testing-by-wavepackets}
    \bibliographystyle{bjoern_style}
\end{document}